\documentclass[12]{amsart}
\usepackage{amsfonts,amssymb,amscd,amsmath,enumerate,verbatim,calc}

\usepackage[all]{xy}

\newcommand{\CM}{Cohen-Macaulay}

\newcommand{\ff}{\text{if and only if}}
\newcommand{\wrt}{with respect to}

\newcommand{\m}{\mathfrak{m} }

\newcommand{\p}{\mathfrak{p} }

\newcommand{\grade}{\operatorname{grade}}
\newcommand{\depth}{\operatorname{depth}}

\newcommand{\htt}{\operatorname{ht}}

\newcommand{\Proj}{\operatorname{Proj}}

\newcommand{\Tor}{\operatorname{Tor}}
\newcommand{\Supp}{\operatorname{Supp}}
\newcommand{\Min}{\operatorname{Min}}

\theoremstyle{plain}

\newtheorem{theorem}{Theorem}[section]
\newtheorem{corollary}[theorem]{Corollary}
\newtheorem{lemma}[theorem]{Lemma}
\newtheorem{proposition}[theorem]{Proposition}

\theoremstyle{definition}
\newtheorem{definition}[theorem]{Definition}

\newtheorem{remark}[theorem]{Remark}
\newtheorem{example}[theorem]{Example}

\theoremstyle{remark}

\begin{document}

 \title{analytic deviation one ideals and Test modules}
 \author{Ganesh S. Kadu and Tony~J.~Puthenpurakal}
\date{\today}

\address{Department of Mathematics,
          College of Engineering Pune, \newline
        Pune 411005, India}
\email{ganeshkadu@gmail.com}
        %(An Autonomous Institute of Government of Maharashtra),

\address{Department of Mathematics, Indian Institute of Technology Bombay, Powai, Mumbai 400 076, India}

\email{tputhen@math.iitb.ac.in}
\begin{abstract}
Let $(A,\m)$ be a \CM \ local ring of dimension $d \geq 1$ and $I$
an ideal in $A.$ Let $M$ be a finitely generated maximal
{\CM}$A$-module. Let $I$ be a locally complete intersection ideal
 with $\htt_M(I)=d-1$, $l_M(I)=d$ and reduction number at most
one.
%If  $\ell(\Tor^{A}_{1}(M, A/I^{n+1})) < \infty$
 We prove that the polynomial $ n \mapsto \ell(\Tor^{A}_{1}(M, A/I^{n+1})) $ either has degree $d-1$ or  $F_I(M) $ is a free $F(I)-$module.
%then there exists a polynomial $t_I^A(M,z) \in \Q[z]$ of degree $\leq l(I)-1$   such that $t_I^A(M,n) = \ell(\Tor^{A}_{1}(M, A/I^{n+1}))$ for $n \gg 0.$
 %We prove that if  $\deg t_I^A(M,z) < d -1$ then
\end{abstract}
\maketitle
\section{introduction}
Let $(A,\m)$ be a \CM \ local ring of dimension $d \geq 1$ and $I$
an ideal in $A.$ Let $M$ be a finitely generated $A$-module.
 Let $F(I)= \bigoplus_{n\geq 0}I^n/\m I^n$ be the fiber cone of $I$ and
 $F_I(M) = \bigoplus_{n\geq 0}I^nM/\m I^nM$ be the fiber module of $M$ \wrt \ $I.$ Let $l(I)=\dim F(I)$ denote the analytic spread of $I.$ \\
 Suppose that $\ell(\Tor^{A}_{1}(M, A/I^{n+1})) < \infty$ for all $n \geq 1$.
 In \cite{Kod} Kodiyalam proved that there exists a polynomial $t_I^A(M,z) \in \mathbb Q [z]$ of degree $\leq l(I)-1$   such that
\[
t_I^A(M,n) = \ell(\Tor^{A}_{1}(M, A/I^{n+1})) \quad \text{for} \ n
\gg 0.
\]
It is of some interest to find the degree of  $t_I^A(M,z)$. In
\cite[18]{Pu1} it was proved that if $M$ is a maximal \CM \
$A$-module and $I = \m$ then
 \[
 \deg t_\m^A(M,z) < d -1 \quad \text{\ff} \quad M \ \text{is free}.
 \]
In \cite[Theorem I]{Pu-Srikanth} this result was generalized to
arbitrary finitely generated modules with projective dimension at
least $1.$
On the other hand it is easily seen that if $I = (x_1,\ldots,x_d)$ is a parameter ideal in $A$ and $M$ is a maximal \CM \ $A-$module then $\Tor^{A}_{1}(M, A/I^{n+1}) = 0$ for all $n \geq 0$, see \cite[20]{Pu1}.\\
 Assume now that $M$ is a non-free maximal {\CM} $A-$module. Suppose that $I$ is not $\m-$primary. Two natural conditions
 when  $\ell(\Tor^{A}_{1}(M, A/I^{n+1})) < \infty$ for all $n \geq 0$ are as follows:\\
 (1) $A_{\p}$ is a regular local ring for all primes $\p \neq \m.$ \\
(2) $\htt(I) =d-1$ and $I$ is locally a complete intersection. \\
 We focus our attention on the second condition.  We assume that $\htt(I) =d-1$ and $l(I)=d$ with $I$ a locally complete intersection ideal. \\
 We prove the following theorem
\begin{theorem}\label{anal}
Let $(A,\m)$ be a \CM \ local ring of dimension $d \geq 1$. Let
$M$ be a maximal  \CM  \ $A$-module. Assume that $I$ is locally a
complete intersection  ideal in $A$ with $r(I) \leq 1, \; \htt (I)
= \htt_M(I)= d -1$ and $l(I) = d.$ If $\deg t_I^A(M,z) < d -1$
then $F_I(M) $ is a free $F(I)-$module.
\end{theorem}
When $(A,\m)$ is a hypersurface ring of dimension $d=1$ we show
that  $\deg t_I^A(M,z) < d -1$  if and only if $M$  is free
$A-$module. We also give an example of a  non-free maximal {\CM}
$A-$module such that $F_I(M)$ is free $F(I)-$module.
\section{Preliminaries}
Let $(A,\m)$ be a  local ring with infinite residue field $k=A/
\m.$ Let $I$ be an ideal in $A$ and  $M$ be a finitely generated
$A$-module.
\begin{lemma}Let $(A,\m)$ be a \CM \ local ring of dimension $d \geq 1$. Let $M$ be maximal \CM \; $A-$module. Assume that $I=(x_1,...,x_m)$ is an ideal in $A$ generated by an $A-$regular sequence. Then $\Tor_1^A(M, A/I^n)=0$ for all $ n \geq 1.$
\end{lemma}
\begin{proof}
 The proof is by induction on $n.$ Let $n=1.$ Let  $L$ be the first syzygy of $M.$ We have the exact sequence,
 \[ 0 \rightarrow L \rightarrow A^s \rightarrow M \rightarrow 0 \]
 Since $A$ is \CM\;and $M$ is maximal \CM \;  $x_1,...,x_m$ is also an $M-$regular sequence. Now by \cite[1.1.4]{BH} tensoring with $A/I$ gives the following  exact sequence,
 \[ 0 \rightarrow L/IL \rightarrow (A/I)^s \rightarrow M/IM \rightarrow 0 \]
 Thus $\Tor_1^A(M, A/I)=0.$ Now assuming that  $\Tor_1^A(M, A/I^n)=0$ for $n=k$ we prove for $n=k+1.$ Consider the following exact sequence
 \[ 0 \rightarrow I^k/I^{k+1} \rightarrow A/I^{k+1} \rightarrow A/I^k \rightarrow 0 \]
 Now since $I=(x_1,...,x_m)$ is genrerated by $A-$regular sequence by \cite[1.1.8]{BH} $I^k/I^{k+1} \cong (A/I)^p$ for some $p \geq 1.$ Now applying the functor $- \otimes A/I^k  $ to the above exact sequence we get the following long exact sequence,
 \[ \rightarrow \Tor_1^A(M, I^k/I^{k+1}) \rightarrow  \Tor_1^A(M, A/I^{k+1}) \rightarrow  \Tor_1^A(M, A/I^k) \rightarrow \]
 Now $ \Tor_1^A(M, I^k/I^{k+1}) \cong  \Tor_1^A(M, (A/I)^p) \cong \Tor_1^A(M, A/I)^p =0.$ By induction hypothesis
 $\Tor_1^A(M, A/I^k)=0.$ Hence $ \Tor_1^A(M, A/I^{k+1})=0.$
\end{proof}
We first give two natural conditions when $\ell(\Tor^{A}_{1}(M,
A/I^{n+1}) < \infty$ for all $n \geq 0.$
\begin{lemma}\label{tor1} Let $(A,\m)$ be a \CM \ local ring of dimension $d \geq 1$. Let $M$ be non-free maximal \CM \; $A-$module and $I$ be ideal in $A.$ Assume $I$ is not $\m$-primary. Then $\ell(\Tor^{A}_{1}(M, A/I^{n+1}) < \infty$ for all $n \gg 0$ if  $A_{\p}$ is a regular local ring for all primes $\p \neq \m.$
\end{lemma}
\begin{proof}
 Suppose that $A_{\p}$ is a regular local ring for all primes $\p \neq \m.$  We first note that if
 $  P \in \Supp(M)$ then
 $M_P$ is maximal \CM. Thus when  $\p \neq \m$ it follows that  either $M_P=0$ or $M_P$ is free. It is now easy to see that
 $ \Supp(\Tor^{A}_{1}(M, A/I^{n+1})) \subseteq \{ \m \} .$  So $\ell(\Tor^{A}_{1}(M, A/I^{n+1}) < \infty$ for all $n \geq 0.$\\
  %In this case we use \cite[Theorem 18]{Pu1}  to see that $\supp(\Tor^{A}_{1}(M, A/I^{n+1})) \subseteq \{ \m \}.$
\end{proof}
\begin{proposition}\label{tor2}
 Let $(A,\m)$ be a \CM \ local ring of dimension $d \geq 1$. Let $M$ be non-free maximal \CM \; $A-$module and $I$ be ideal in $A.$  If  $ \htt(I) =d-1$, $l(I)=d$ and $I$ is locally a complete intersection then $\ell(\Tor^{A}_{1}(M, A/I^{n+1}) < \infty$ for all $n \gg 0.$
\end{proposition}
\begin{proof}
It suffices to show that  $\Supp(\Tor^{A}_{1}(M, A/I^{n+1}))
\subseteq \{ \m \}.$  So let $ P \in \Supp(M/IM)$ and consider
 $(\Tor^{A}_{1}(M, A/I^{n+1}))_P = \Tor^{A_P}_{1}(M_P, A_P/I_P^{n+1}).$ If $P \neq \m$ then by hypothesis $I_P$ is a complete intersection. If $\dim M \geq 2$ then by \cite[20]{Pu1} we have  $\Tor^{A_P}_{1}(M_P, A_P/I_P^{n+1})=0.$ \\ If $\dim M = 1= \dim A$ then $\dim M_P= \dim A_P =0.$
  So when $ P \neq \m$ we have $I_P^n=0$ for $n \gg 0.$  Hence
 $$\Tor^{A_P}_{1}(M_P, A_P/I_P^{n+1})= \Tor^{A_P}_{1}(M_P, A_P)=0$$
  So $\Supp(\Tor^{A}_{1}(M, A/I^{n+1})) \subseteq \{ \m \}.$ Thus
  $ \ell(\Tor^{A}_{1}(M, A/I^{n+1}) < \infty \; \text{for all} \; n \gg 0. $
  \end{proof}
\begin{remark}
 If $l(I) =d-1$ and $I$ is locally a complete intersection then $I$ is a complete intersection.
\end{remark}
 We need the following lemma.
 \begin{lemma}\label{free}Let $R = \bigoplus_{n\geq 0} R_n$ be a \CM \ standard graded algebra over an infinite field $k = R_0$. Let $M=\bigoplus_{n\geq 0} M_n$ be a finite graded $R$-module.
Let $x \in R_1$ be $M$-filter regular. Set $N = M/xM$ and $S =
R/(x)$. Assume $\depth(R) \geq  1.$ Then
\begin{enumerate}[\rm (1)]
\item We can choose $x$ such that $x$ is  $R$-regular. \item If
$\grade(R_+, N) > 0$ then $x$ is $M$-regular. \item If $N$ is free
$S$-module and $x$ is $M$-regular  then $M$ is a free $R$-module.
\end{enumerate}
 \end{lemma}
 \begin{proof}
  (i) We first observe that since  $\depth(R) \geq 1$ we have $R_+ \notin Ass(R).$ Let $B= Ass(R) \cup Ass(M) \setminus V(R_+).$
  Note that $R_1$ is a vector space over an infinite field $k$. The set $C= \{ P \cap R_1 \mid P \in B \}$ is a finite set consisting of proper subspaces of $R_1.$ So we can choose an  element $x \in R_1 \setminus \bigcup_{P \in C}P \cap R_1.$ It is now easy to see that $x$ is $R-$regular and $M-$filter regular.\\
  (ii) Let $N=M/xM$ and $B=(0:_M x).$ Since $B_n=0$ for $n>>0,$ $H^i_{R+}(B)=0$ for $i>0.$ So we have the following long exact sequence

$$0 \xrightarrow[\hspace{13pt}]{} H^0_{R+}(B) \xrightarrow[\hspace{13pt}]{} H^0_{R+}(M)(-1) \xrightarrow[\hspace{13pt}]{\mu_x} H^0_{R+}(M) \xrightarrow[\hspace{13pt}]{} H^0_{R+}(N) $$
 $$ ~~~~~~~~~~~~~~~~~~~\; \; \; \; \; \xrightarrow[\hspace{13pt}]{} H^1_{R+}(M)(-1) \xrightarrow[\hspace{13pt}]{\mu_x}  H^1_{R+}(M)    \xrightarrow[\hspace{13pt}]{} H^1_{R+}{R+}(N)  $$
 \[ ..... \]
 Since $\grade(R_+, N) > 0$ we have $H^0_{R+}(N)=0.$ Hence the exact sequence above becomes
 $$0 \xrightarrow[\hspace{13pt}]{} H^0_{R+}(B) \xrightarrow[\hspace{13pt}]{} H^0_{R+}(M)(-1) \xrightarrow[\hspace{13pt}]{\mu_x} H^0_{R+}(M) \longrightarrow 0$$
  Now note that $H^0_{R+}(M) $ is a finite dimensional vector space. Thus $\mu_x$ surjective gives
  $H^0_{R+}(M) \cong H^0_{R+}(M)(-1).$ So $H^0_{R+}(M) =0.$ Therefore  $\grade(R_+, M) > 0.$ It follows that $x$ is $M-$regular.\\
  (iii) Let $l=\mu(M).$ Note that $l =\mu(M/xM).$
  Now consider the following exact sequence:
  $$ 0 \longrightarrow L \longrightarrow R^l\longrightarrow M \longrightarrow 0 $$
  where $L$ is the first syzygy of $M.$ Now since $x$ is $M-$regular it follows from  \cite[1.1.4]{BH} that
  $$0 \longrightarrow  L/xL \longrightarrow  S^l \longrightarrow N \longrightarrow 0 $$
  is exact sequence. By hypothesis $N$ is free $S-$module and so $N \cong S^l.$ Thus $L/xL=0.$ By graded Nakayama lemma
  %Since $L$ is graded submodule of $R^k$ and $x\in R_1$
   we get that $L=0.$ Thus $M \cong R^l$ and hence a free $R-$module.
  %Since  $N$ is free $S-$module we have $ \projdim_S N =0.$  Since $x$ is $R-$regular and $M-$regular it follows that $ \projdim_R M = \projdim_S N.$ So $M$ is free $R-$module.
 \end{proof}
 We recall the definition of superficial element.
 \begin{definition}An element $ x \in I$ is $I-$superficial for $M$ if there exists a positive integer $c$ with
\[(I^{n+1}M :_M x) \cap I^cM = I^nM  \; \; \text{for all} \; \;  n \geq c. \]

\end{definition}
\begin{remark}\label{super1}
 Assume that $\grade(I,M) \geq 1.$  If an element is $I-$superficial on $M$ then $x$ is regular on $M$ and
 \[(I^{n+1}M :_M x) = I^nM  \; \text{for all} \; n>>0. \]
\end{remark}

\noindent A convention: The degree of the zero polynomial is
defined to be $- \infty.$

% We state the following Lemma from \cite[Lemma 2.5]{Pu-Srikanth} which is easily adapted to our case.
\begin{lemma}\label{Pu}
 Let $(A, \m) $ be a local ring and $M$ a finite non-free $A-$module. Let $I$ be an ideal in $A.$  Denote by $L$ the first syzygy of $M.$ Let $x \in I$ be a superficial non-zero divisor on $A,$ $M,$ and $L.$  Suppose that $\ell(\Tor^{A}_{1}(M, A/I^{n+1})) < \infty $ for $n >> 0.$
 Set $B=A/xA $, $N=M/xM$ and $J = I/(x).$ Then  we have
$$t_I^A(M;n) =  t_I^A(M; n-1) + t_J^B(N; n) ~~~~ for ~ all~~~ n \gg 0. $$
$$\deg t_J^B(N;z) \leq \deg t_I^A(M;z) - 1. $$
\end{lemma}
\begin{proof}
 Since $x$ is $I-$superficial for $A,$ one has following exact sequence for all $n \gg 0$
$$0 \longrightarrow  A/I^n  \xrightarrow[\hspace{13pt}]{i} A/I^{n+1} \longrightarrow  B/J^{n+1} \longrightarrow 0$$
 where the map $i$ is defined by $i_n(a+I^n)=xa+I^{n+1}.$ Applying $M \otimes_A-$ to above exact sequence gives the following exact sequence of $A-$modules

 \[ \Tor_1^A(M, A/I^n) \xrightarrow[\hspace{13pt}]{\Tor_1^A(M, i_n)} \Tor_1^A(M, A/I^{n+1}) \xrightarrow[\hspace{13pt}]{} \Tor_1^A(M, B/J^{n+1})  \xrightarrow[\hspace{13pt}]{}\]
 \[  M/I^nM \xrightarrow[\hspace{13pt}]{M\otimes i_n} M/I^{n+1}M  \xrightarrow[\hspace{13pt}]{} N/J^{n+1}N \xrightarrow[\hspace{13pt}]{} 0.\]
    Since $x$ is $I-$superficial on $M$ the map $M \otimes i_n$ is injective for $n \gg 0.$ We claim that
     the map $\Tor_1^A(M, i_n)$ is injective for $n \gg 0.$ For this consider the exact sequence defining $L,$ i.e.
     \[0 \longrightarrow L \longrightarrow F \longrightarrow  M \longrightarrow 0. \]
     Now applying the functors $-\otimes_A A/I^n$ and $-\otimes_A A/I^{n+1} $ one gets the  following commutative diagram

     \[
  \xymatrix
{
 0
 \ar@{->}[r]
  & \Tor_1^A(M,A/I^n)
    \ar@{->}[d]^{\Tor_1^A(M,i_n)}
\ar@{->}[r]^{}
 & L/I^nL  \ar@{->}[d]^{L \otimes i_n}
 \\
 0
 \ar@{->}[r]
  & \Tor_1^A(M,A/I^{n+1})
\ar@{->}[r]
 & L/I^{n+1}L
  }
\]
Notice the following \[ Ker(L \otimes i_n) =
\frac{I^{n+1}L:x}{I^nL}. \] As $x$ is superficial on $L$ it
follows that the map $L \otimes i_n$ is injective for $n \gg 0.$
Thus
$ \Tor_1^A(M,i_n)$ is injective for $n \gg 0.$ \\
 So for $n \gg 0$ above long exact sequence becomes

\[0 \longrightarrow  \Tor_1^A(M, A/I^n) \longrightarrow   \Tor_1^A(M, A/I^{n+1}) \longrightarrow   \Tor_1^A(M, B/J^{n+1}) \longrightarrow 0. \]
Now since $x$ is both $R-$regular and $M-$regular we obtain the
following isomorphism, see \cite[18.2 ]{Mat}
$$ \Tor_1^A(M, B/J^{n+1}) \cong  \Tor_1^B(N, B/J^{n+1}).$$
From this isomorphism and the exact sequence above it follows that
for $n \gg 0$
\[\ell_A( \Tor_1^A(M, A/I^{n+1})) = \ell_A( \Tor_1^A(M, A/I^n)) + \ell_A(\Tor_1^B(N, B/J^{n+1})).\]
Thus it follows that
$$\deg t_J^B(N;z) \leq \deg t_I^A(M;z) - 1. $$

\end{proof}

Recall that an ideal $J \subseteq I$ is reduction of $I$ if there
exists a natural number $m$ such that $JI^n=I^{n+1}$ for all $ n
\geq m.$ We define $r_J(I)$ to be the least such $m.$ A reduction
$J$ of $I$ is called minimal if it is minimal with respect to
inclusion. Reduction number of $I$  is defined as follows,
$$r(I)= \min \{ r_J(I) \mid  J ~ is ~ minimal ~ reduction ~ of ~ I  \}.$$

 \begin{proposition}\label{regg1}
  Let $(A,\m)$ be a \CM \ local ring of dimension $d \geq 2$ with infinite residue field $k = A/ \m.$ Let $M$ be a maximal  \CM  \ $A$-module. Assume
$I $ is locally a complete intersection  ideal in $A$ with $\htt
(I) = \htt_M(I)= d -1$, $r(I) \leq 1$ and $l_M(I) = d$. Denote by
$L$ the first syzygy of $M.$ Then we can choose $x \in I$
satisfying the following conditions
\begin{enumerate}[\rm (i)]
 \item $x^*$ is $G_I(A)-$regular.
 \item$x^*$ is $G_I(M)-$regular.
 \item $x^*$ is $G_I(L)-$filter regular.
 \item $x^o$ is $F(I)-$regular.
 \item $x^o$ is $F_I(M)-$ filter regular.
\end{enumerate}
\end{proposition}
\begin{proof}
Since  $r(I) \leq 1$ by \cite[4.2]{Corta} $F(I)$ is {\CM}. Note
here that
 $l_M(I)=d$ gives $l(I)=d.$
 Now by \cite[2.9]{Hun} $G_I(A)$ is {\CM} ring. Extension of Huckaba-Huneke's result to maximal \CM\;modules gives $G_I(M)$ is a {\CM}  $G_I(A)$ module. Let $\mathfrak A = Ass(G_I(A)) \cup Ass(G_I(M)).$ Note that since $ d \geq 2$ we have  $\grade(I) = d -1 \geq 1.$
 So if $ P \in Ass(G_I(A)) \cup Ass(G_I(M))$ then $ P \notin V(G(I)_+).$ Hence
 \[ \mathfrak A  =  Ass(G_I(A)) \cup Ass(G_I(M)) \setminus V(G(I)_+). \]
 Consider the following set
 $$\mathfrak B =  \mathfrak A \cup Ass(G_I(L)) \setminus V(G(I)_+) . $$
 Let $ \mathfrak B = \{ P_1, ..., P_t \}.$ Set $V_i = P_i \cap I/I^2.$ It is clear that $V_i$ are proper submodules of $I/I^2.$
 Since $F(I)$ is \CM\;we have
 \[ \mathfrak C = Ass(F(I)) = Ass(F(I)) \setminus  V(F(I)_+). \]

 Now let $\mathfrak D = \mathfrak C \cup Ass(F_I(M)) \setminus V(F(I)_+) = \{ Q_1,...,Q_k \}.$ Set $W_i = Q_i \cap I/\m I.$ Observe that $V_i \otimes k$ and $ W_j$ are proper subspaces of $I/ \m I$ for $1 \leq i \leq t $ and $1 \leq j \leq k.$ Since $k$ is infinite we can choose \[ \bar x \in \frac{I}{\m I} \setminus \{ \bigcup_{i=1}^t V_i\otimes k , \; \bigcup_{i=1}^k W_i \}. \]
 We now verify conditions (i) to (v). Since $ x \notin V_i \otimes k$ gives that $ x \notin V_i.$ Therefore $x^* \notin P$ for $ P \in \mathfrak A.$ Hence $x^*$ is $G_I(A)-$regular and $G_I(M)-$regular. So (i) and (ii) follow. Similarly (iii) follows because $x^* \notin P$ for $P \in \mathfrak B.$ For (iv) and (v) we notice that $x^o \notin P$ for $P \in \mathfrak D$ since $x \notin W_i.$ It therefore follows that $x^o$ is  $F(I)-$regular and $F_I(M)-$ filter regular.
\end{proof}
\begin{corollary}\label{super-reg} Let $(A,\m)$ be a \CM \ local ring of dimension $d \geq 2$ with infinite residue field $k = A/ \m.$  With $I$, $L$, $M$ and $ x \in I$ as in the  Proposition \ref{regg1} we have
\begin{enumerate}[\rm (i)]
 \item $x$ is $I-$superficial on $A$, $M$ and $L.$
 \item $x$ is regular on $A$, $M$ and $L.$
\end{enumerate}
\end{corollary}
\begin{proof}
 (1) This follows easily from the first three conditions of the  Proposition \ref{regg1}. \\
 (2) Note that grade of $I$ w.r.t. $R$, $M$ and $L$ are all at least $1.$ So superficial elements on them are regular elements.
\end{proof}
\begin{remark}\label{gen-complete} For $x \in I $ as chosen in Proposition \ref{regg1} we have $\bar I = I/(x)$ is locally a complete intersection ideal in $\bar A  =A/(x).$ To see this it suffices to observe that $\Min(\bar A/ \bar I)= \{ P/(x) \mid P \in Min(A/I) \}.$

\end{remark}

\begin{remark}
 If $J$ is minimal reduction of $I$ such that $J$ is genrerated by a regular sequence
 and $JI = I^2$ then $F(I)$ is \CM, see \cite[1]{SK}. This result holds true even when $I$ is locally a complete intersection, see \cite[4.2]{Corta}.
\end{remark}
\begin{lemma} \label{iso}Let $x \in I$ and set $B=A/(x)$, $N = M/ xM$ and $ \bar I= I /(x).$ Suppose $x^*$ is $G_I(M)-$regular then  we have the following isomorphism
$$ F_{\bar I}(N) \cong \frac{F_I(M)}{x^oF_I(M)}.$$
% \quad and \quad  F(\bar I) \cong \frac{F(I)}{x^oF(I)} $$
\end{lemma}
\begin{proof}
Observe that we have $$F_{\bar I}(N)_n=\frac{I^nM+(x)M}{\m I^nM
+(x)M} \; \; \textit{and}
 \; \; [ \frac{F_I(M)}{x^oF_I(M)} ]_n =\frac{I^nM}{xI^{n-1}M+ \m I^nM} $$
 Consider the following natural map
 \[\alpha: I^nM \longrightarrow  \frac{I^nM+(x)M}{\m I^nM +(x)M} \]
 We now have \[ Ker \alpha = I^nM \cap (\m I^nM +(x)M) =\m I^nM + (x)M \cap I^{n}M\]
 Since $x^*$ is $G_I(M)-$regular for all $n \geq 1$ we have $(x)M \cap I^{n}M = xI^{n-1}M$ by Valabrega-Valla criterion, see \cite[2.6]{VV}.
 Hence  $Ker \alpha = xI^{n-1}M+ \m I^nM$  and so \[\frac{I^nM+(x)M}{\m I^nM +(x)M} \cong \frac{I^nM}{xI^{n-1}M+ \m I^nM} \]
 Hence \[ F_{\bar I}(N) \cong \frac{F_I(M)}{x^oF_I(M)}\]
 %Thus \[F(\bar I) \cong \frac{F(I)}{x^oF(I)}\]
 %Similar argument shows that
\end{proof}

\section{Proof of the Main Theorem}

 We now prove the main theorem.
 \begin{theorem}\label{anal}
Let $(A,\m)$ be a \CM \ local ring of dimension $d \geq 1$. Let
$M$ be a maximal  {\CM} $A$-module. Assume $I $ is locally a
complete intersection  ideal in $A$ with $r(I) \leq 1, \; \htt (I)
= \htt_M(I) = d -1$ and $l_M(I) = d$. If $\deg t_I^A(M,z) < d -1$
then $F_I(M) $ is a free $F(I)-$module.
 \end{theorem}
 \begin{proof}
 It follows from Proposition \ref{tor2} that $\ell(\Tor^{A}_{1}(M, A/I^{n+1})) < \infty$ for all $n \geq 0.$ So $\deg t_I^A(M,z) \leq d -1.$ Suppose now that
 $\deg t_I^A(M,z) < d -1.$  Note that $l_M(I)=d$ gives $l(I)=d.$ We prove the result by induction on $d.$\\
 \indent  Suppose first $d=1.$  Let $J=(a)$ be a minimal reduction of $I.$ Since $r(I) \leq 1$ by \cite[4.2]{Corta} $F(I)$ is {\CM}.  Now since $r(I) \leq 1$ and $\dim F(I) =1$ it follows from \cite[4.1.12]{BH} that $\mu(I^n) =c$ for all $n \geq 1.$
  %So only the first local cohomology module is non-zero. By \cite[Theorem 4.3.5]{BH} we have for all $n \geq 0$
 % \[ \mu(I^n)-c = - \dim_k [H^1_{F(I)+}(F(I))]_n   \]
 % Note that $\mu(I^n) =c$ for all $n \gg 0.$ Let $n_0$ be a non-negative integer such that $\mu(I^n)=c$ for all $ n \geq n_0$ and $ \mu(I^{n_0-1}) \neq c.$ This gives us that
 % $$\dim_k [H^1_{F(I)+}(F(I))]_{n_0-1} \neq 0.$$ Hence $\mu(I^{n_0-1}) < c.$ If $ n_0 \geq 2 \Longrightarrow I^{n_0}=aI^{n_0-1}.$ It follows that $\mu(I^{n_0}) \leq \mu(I^{n_0-1}) < c $ which is a contradiction. So $n_0 \leq 1.$ Thus we have $ \mu(I^n)=c$ for all $ n \geq 1.$
  %In this case we have $\grade(I, M)=0 $ and $ I =(a)$  where
 %$a$ is a zero-divisor on $M.$ Also note that $ I^n =(a^n)$ for all $n.$
 Since $\deg t_I^A(M,z) < 0 $ we have $\Tor^{A}_{1}(M, A/I^{n}) = 0$ for $ n \gg  0.$ So $M\otimes_A I^n = I^nM.$ Tensoring with $A/\m$ gives the following  $$\frac{M}{\m M} \otimes_k \frac{I^n}{\m I^n} = \frac{I^nM}{\m I^n M} .$$
  Therefore
  \begin{equation}\label{length}
   \ell \left(\frac{I^nM}{\m I^n M} \right) = c \mu(M) \quad \ \text{for all} \  n \gg 0.
  \end{equation}
  \textbf{Claim:} $F_I(M)$ is a free $F(I)$-module.\\
  Note that we have $aI^nM = I^{n+1}M$ for $n \geq 1.$ Therefore
  \[c \mu(M) \geq  \mu(IM) \geq \mu(I^2M) \geq \ldots \geq \mu(I^nM) \geq \ldots \]
  But we have $\mu(I^nM) =c \mu(M)$ so that $ \mu(I^nM)= c \mu(M)$ for all $n \geq 1.$\\
  Let $\mu(M) =t$ and $M=< m_1,...,m_t>.$ We define a graded $F(I)-$linear map $\phi$
  \[\phi :F(I)^t \longrightarrow F_I(M) \]
  by $\phi(e_i) =\bar m_i$ where $ \bar m_i \in M/\m M.$
  Observe that $\phi$ is surjective,
   $$ \dim_k [F_I(M)^t]_0 = t $$
   $$ \dim_k [F_I(M)]_0 = t $$
   and for $n \geq 1$
  $$ \dim_k [F(I)^t]_n = ct  $$
  $$ \dim_k [F_I(M)]_n = ct $$
  Therefore $\phi$ is an isomorphism. So $F_I(M)$ is free $F(I)-$module.\\
  %Notice $F(I) \cong k[X].$ Set $u = a + \m I$ the image of $a$ in $F(I)_1$.
  %By graded analogue of modules over PID we get the following graded isomorphism
%F_I(M) \cong  F(I)^s \bigoplus\left( \bigoplus_{n = 1}^{r} \frac{F(I)}{u^{n_i} F(I)}\right) \quad \text{where} \ n_1 \leq n_2 \leq\ldots\leq n_r.
%\]
 %Since all the generators of $F_I(M)$ are in degree zero we see that $$\mu_A(M)= \mu_{F(I)} \big( F_I(M) \big).$$ By
 %graded isomorphism given above we get $\mu_{F(I)} \big( F_I(M) \big) = s+r.$  By the graded isomorphism given above it follows that multiplicity of $F_I(M)$ is $s.$ But by (\ref{length}) multiplicity of $F_I(M)$ is  $  \mu_A(M).$ So $ r=0.$ Hence $F_I(M)$ is a free $F(I)$-module.\\
 Suppose now that $ d \geq 2.$  So $ \htt(I) = d-1 \geq 1.$ Choose $x \in I  $ as in Proposition \ref{regg1}.
  So now $x$ satisfies all the conditions given in Proposition \ref{regg1} and Corollary \ref{super-reg}. Set $B=A/(x)$, $ N = M/ xM$ and $ \bar I= I /(x).$ By our choice of $x$ we have  $\htt(\bar I) = \htt_N(\bar I) = d-2$ and  $l_N(\bar I)=d-1.$  By Remark \ref{gen-complete} $\bar I$ is locally a complete intersection in $B.$ Also $r(\bar I) \leq 1.$ Now by Lemma \ref{Pu} we have
  $$\deg t_{\bar I}^B(N;z) \leq  \deg t_I^A(M;z) - 1.$$
 By our hypothesis $\deg t_I^A(M;z) < d-1$ and so we have $\deg t_{\bar I}^B(N;z) < d-2.$
 Thus by induction hypothesis we have $F_{\bar I }(N)$ is free
   $F(\bar I)-$module. Note that we have following isomorphisms by  Lemma \ref{iso}

   $$ F_{\bar I}(N) \cong \frac{F_I(M)}{x^oF_I(M)} \quad and \quad  F(\bar I) \cong \frac{F(I)}{x^oF(I)} $$
   Since $ F_{\bar I}(N) $ is free $ F(\bar I)-$module one has $\depth( F_{\bar I}(N)) = \depth(F(\bar I)).$
   Now since $F(I)$ is \CM \; and $x^o$ is $F(I)-$regular we have $\depth(F(\bar I)) =d-1 \geq 1.$ Hence
   $\depth( F_{\bar I}(N)) \geq 1.$
   By our choice of $x$ we have that $x^o$ is $F_I(M)-$filter regular. So $x^o$ is $F_I(M)-$regular by Lemma \ref{free}(2). Again by Lemma \ref{free}(3) we get that $F_I(M)$ is free $F(I)-$module.

 \end{proof}
 \section{The case of hypersurface}
 When $A$ is a hypersurface ring we show that if $M$ is non-free maximal \CM\;$A-$module of dimension $d=1$ having constant rank  then $\deg t_I(M, z) = d-1.$

\begin{proposition}
 Let $A$ be a hypersurface ring of dimension $d = 1.$  Let $M$ be a maximal \CM\;$A-$module having constant rank. Let $I$ be a locally  complete intersection ideal with $\htt(I) =0$ and $l(I)=1.$ Then the following conditions are equivalent:
 \begin{enumerate}
  \item [\rm (i)] $\deg t_I^A(M, z) < 0.$
  \item [\rm (ii)] $M$ is free $A-$ module.
 \end{enumerate}
\end{proposition}
\begin{proof}
  So let $d=1$ first.
 Consider the following exact sequence
 \[0  \longrightarrow I^n \longrightarrow A \longrightarrow A/I^n \longrightarrow 0 \]
 Tensoring with $M$ we get
 \[0  \longrightarrow \Tor_1^A(M, A/I^n) \longrightarrow M \otimes_A I^n \longrightarrow M \longrightarrow  M/ I^nM \longrightarrow 0\]
 Suppose $\deg t_I(M, z) < d-1=0.$ So $\Tor_1^A(M, A/I^n) =0. $ Since $\dim M=1$ we obtain from the exact sequence above that $M \otimes_A I^n$ is maximal \CM\;$A-$module. It now follows from the \cite[Theorem 3.1]{HunWei} that atleast one of  $M$ or $I^n$ is free $A-$module. But $I^n$ cannot be free because $\htt(I) =0.$ Hence $M$ is free $A-$module.  \\
 (ii) $\Rightarrow$ (i) is obvious.
\end{proof}
%\begin{remark}
 %By the proposition above it follows that if $M$ is non-free maximal \CM\; module over hypersurface $A$ then $\deg t_I(M, z) = d-1.$
%\end{remark}

 \section{Example}

 We give an example where $F_I(M)$ is one dimensional non-free $F(I)-$module having depth zero.
 \begin{example}  Let $A$ be a \CM \; local ring of dimension $1$ with atleast two distinct minimal prime ideals.
Let $P_1$ and $P_2$ be those two minimal prime ideals.  Now let
$M=A/P_1 \oplus A/P_2.$ Notice $M$ is maximal {\CM} $A-$module of
dimension $1.$ Now choose $b \in P_2 \setminus P_1.$ Let $ J =
(b)$ and $ \mathfrak A = \Min(A/J)=\{P_2,P_3,...,P_s\}.$ Since
$A_{P_i}$ is Artinian we have $s_ib^{n_i}=0$ for $ s_i \notin P_i$
and for $1 \leq i \leq s.$  Let $ n = max \{ n_i \mid  1 \leq i
\leq s \}.$ Let $a = b^n$ and $I =(a).$ We now have $I_{P}=0$ for
$ P \in \Min(A/I).$ We thus have
  \begin{enumerate}
 \item [\rm(i)] $ \htt(I) = 0.$
 \item [\rm(ii)] $l(I)=1.$
 \item [\rm(iii)] I is locally a complete intersection.
 \item [\rm(iv)] $\mu (M) =2.$
 \item [\rm(v)] $\mu(a^nM) = \ell(a^nM/\m a^nM) = 1$ for all $n \geq 1$ (Since $a^n \notin P_1$ for all $n \geq 1$).
 \item [\rm(vii)] $M$ is not free.
\end{enumerate}
Computation of Tor:\\
$$ \Tor^{A}_{1}(M, A/I^{n}) = \Tor^{A}_{1}(A/P_1\oplus A/P_2, A/(a^{n})) \cong  \frac{P_2 \cap (a^n)}{P_2(a^n)}$$
 Since $a \in P_2 $ we have $P_2 \cap (a^n)= (a^n).$
 So $$ \Tor^{A}_{1}(M, A/I^{n}) \cong \frac{(a^n)}{P_2 (a^n)}$$
 By Nakayama Lemma $(a^n) \neq P_2 (a^n)$ so that for $n \gg 0$ we have $$\ell(\Tor^{A}_{1}(M, A/I^{n})) =c $$ where $c > 0.$
Notice now that $F(I) \cong k[x].$ We claim that $F_I(M)$ is  not
free. Supposing otherwise we get that $F_I(M) \cong k[x]^t$ for
some $t \geq 1.$ Since all the generators of $F_I(M)$ are in
degree zero we get that $\mu(M)=1$ which is contradiction to (iv)
above. So $F_I(M)$ is not free $F(I)-$module. Since $k[x]$ is
regular ring we have $\Proj F_I(M) < \infty.$ By
Auslander-Buchsbaum formula we have $\depth(F_I(M))=0.$
 So $F_I(M)$ is not {\CM} $F(I)-$module.
\end{example}

\bibliographystyle{amsplain}
\bibliography{ref}
\end{document}